\documentclass[a4paper,12pt]{amsart} 
\usepackage{etex,soul}
\usepackage{pgf,ulem}
\usepackage{amscd,amsmath,amssymb,amstext,amsfonts,amsthm,latexsym}
\usepackage{verbatim}
\usepackage[english]{babel}
\usepackage[latin1]{inputenc}
\usepackage{graphicx}
\usepackage{mathrsfs,ulem}
\usepackage{enumitem}
\usepackage[all]{xy}
\usepackage{supertabular,multirow,,hhline}
\usepackage{fancyhdr}
\usepackage{url} 
\usepackage{color, colortbl}

\usepackage{hyperref}

\newtheorem{theorem}{Theorem}[section]
\newtheorem*{thm}{Theorem}
\newtheorem{lemma}[theorem]{Lemma}
\newtheorem{corollary}[theorem]{Corollary}
\newtheorem{proposition}[theorem]{Proposition}
\newtheorem*{conjecture}{Conjecture} %[theorem]

\theoremstyle{remark}
\newtheorem{remark}[theorem]{Remark}

\theoremstyle{definition}
\newtheorem{definition}[theorem]{Definition}%[section]

\DeclareMathOperator{\Sing}{Sing}

\DeclareMathOperator{\Aut}{Aut}

\DeclareMathOperator{\kod}{kod}

\begin{document}

\title[Bloch's conjecture for GBT surfaces]{Bloch's conjecture for Generalized Burniat Type surfaces with $p_g=0$}
\author[I. Bauer, D. Frapporti]{Ingrid Bauer, Davide Frapporti }

\address {Mathematisches Institut der Universit\"at Bayreuth\\
NW II,  Universit\"atsstr. 30\\
95447 Bayreuth}

\thanks{The present work took place in the realm of the DFG
Forschergruppe 790 ``Classification of algebraic
surfaces and compact complex manifolds''.}

\date{\today}
 %%www.ams.org/msc 

\makeatletter
\renewcommand\theequation{\thesection.\arabic{equation}}
\@addtoreset{equation}{section}
\makeatother

\begin{abstract}
The aim of this article is to prove Bloch's conjecture, asserting that the group of rational equivalence classes of zero cycles of degree 0 is trivial for surfaces with geometric genus zero, for regular generalized Burniat type surfaces. The technique is the method of ``enough automorphisms" introduced by Inose-Mizukami in a simplified version due to the first author.

\end{abstract}
\maketitle
%\tableofcontents

%%%%%%%%%%%%%%%%%%%%%%%%%%%%%%%%%%%%%%%%%%%%%%%%%%%%%%%%%%%%%%%%%%%%%%%%%%%%%%%%%%%%%%%%%%%%%%%%%%
\section*{Introduction}
%%%%%%%%%%%%%%%%%%%%%%%%%%%%%%%%%%%%%%%%%%%%%%%%%%%%%%%%%%%%%%%%%%%%%%%%%%%%%%%%%%%%%%%%%%%%%%%%%%%

Let $S$ be a smooth projective surface and let
$$A_0(S)=\bigoplus_{i=- \infty}^\infty A^i_0(S)$$
be the group of rational equivalence classes of zero cycles on $S$. Then
\textit{Bloch's conjecture} asserts the following:

\begin{conjecture}[{\cite{Bloch}}]
Let $S$ be a smooth surface with $p_g(S)=0$. Then the kernel $T(S)$
of the natural morphism:
$$A_0^0(S)\longrightarrow \mathrm{Alb}(S)$$
is trivial.
\end{conjecture}

The conjecture has been proven for surfaces $S$ with Kodaira dimension $\kod(S)\leq 1$ by 
Bloch, Kas and Liebermann (cf. \cite{bkl}) and has been verified for several examples,  
see e.g. \cite{barlow85b}, \cite{BauerInoue}, \cite{chan2013}, \cite{IM79}, \cite{Keum88}, \cite{Voisin92}.
Thanks to a result of S. Kimura (cf. \cite{kimura05}), all product quotient surfaces (i.e.
minimal models of $(C_1\times C_2)/G$, where $G$ is a finite group acting on the product
of two curves of genus at least 2) with $p_g=0$ satisfy Bloch's conjecture (cf. \cite{BCGP12}).

\textit{Burniat surfaces}   are surfaces of general type with invariants $p_g=0$ and $K^2 = 6,5,4,3,2$
whose  birational models  were constructed by P. Burniat in 1966 (cf. \cite{burniat}) 
 as singular bidouble covers of the projective plane.
In 1994,  M. Inoue (cf. \cite{inoue}) reconstructed them as quotients of a divisor in a product of three elliptic curves by 
a finite group acting freely (see also \cite{BC12} and \cite{BC13}).

Following and generalizing Inoue's approach,
 in the recent paper \cite{BCF14}, we construct and classify  
  a new class of surfaces of general type \textit{``generalized Burniat type surfaces"}. 
  These surfaces   have invariants $K^2=6$ and $0\leq p_g=q\leq 3$ and have been constructed
as quotient of a divisor of multi-degree $(2,2,2)$ in a product of three elliptic curves by 
a free $(\mathbb Z/2\mathbb Z)^3$-action (see Section \ref{GBTS}). 
Generalized Burniat type surfaces form 16 irreducible families; four families have $p_g=0$
 and  form four connected components $\mathfrak{N}_i$
of the Gieseker moduli space $\mathfrak{M}^{can}_{1,6}$. Each component is irreducible, generically
smooth, normal and unirational 
of dimension 4 in two cases and 3 in the others.

The main result of this note is to show that Bloch's conjecture holds for generalized Burniat type 
surfaces with $p_g=0$. The proof uses the method of ``enough automorphisms''
introduced by Inose and Mizukami (cf. \cite{IM79}) and refined by Barlow (cf. \cite{barlow85b}),
but in a simpler way (cf. \cite{BauerInoue}).

\begin{thm}
Let $S$ be a generalized Burniat type surface with $p_g(S)=0$. Then
 $S$ verifies  Bloch's conjecture: $T(S)=A_0^0(S)=0$.
\end{thm}

In \cite{BCF14} the authors show among others the following result:
\begin{thm}
Let $S$ be any surface whose moduli point lies in the connected component of the Gieseker moduli space of surfaces of general type of a regular generalized Burniat type surface, then $S$ is a generalized Burniat type surface.
\end{thm}

More precisely, in \cite{BCF14} it is proven that generalized Burniat type surfaces form exactly four irreducible connected components in their moduli space (cf. Theorem \ref{BCF}).
Therefore our result proves 
Bloch's conjecture for each surface in the  connected component of any regular generalized Burniat type surface.

\section{Bloch's conjecture for surfaces with a $(\mathbb Z/2 \mathbb Z)^2$-action}

The aim of this note is to prove Bloch's conjecture for generalized Burniat type 
surfaces with $p_g=q=0$. The proof uses the method of ``enough automorphisms''
introduced by Inose and Mizukami (cf. \cite{IM79}) and refined by Barlow (cf. \cite{barlow85b}).

\begin{definition}
Let $G$ be a finite group and $H\leq G$ be a subgroup. Then we set
$$z(H):=\sum_{h\in H} h \in \mathbb C G$$
\end{definition}

\begin{lemma}[\cite{barlow85b}]
Let $S$ be a nonsingular surface and $G$ a finite subgroup 
of $\Aut(S)$.
Let $H, H_1,\ldots, H_r$ be subgroups of $G$. We denote by $\mathcal I$ the two-sided ideal
of $\mathbb{C}G$ generated by $ z(H_1),\ldots, z(H_r)$. Assume that
\begin{itemize}[leftmargin=.9cm]
\item[i)] $z(H)\in \mathcal{I }$,
\item[ii)] $T(S/H_i)=0$, for every $i \in \{1,\ldots, r\}$.
\end{itemize}
Then $T(S/H)=0$.
\end{lemma}

Using this result, the first author  proved the following

\begin{proposition}[{\cite[Proposition 1.3]{BauerInoue}}]
Let $S$ be a surface of general type with $p_g(S)=0$. Assume that 
$G=(\mathbb Z/ 2\mathbb Z)^2 \triangleleft \mathrm{Aut}(S)$. Then $S$ satisfies Bloch's conjecture
if and only if for each $\sigma \in G\setminus\{0\}$ the quotient $S/\sigma$ satisfies 
Bloch's conjecture.
\end{proposition}

\begin{remark}
Note that $S/\sigma$ is a surface with at most nodes as singularities and denoting by $X_\sigma\rightarrow S/\sigma$ 
the resolution of its singularities, $X_\sigma$ is minimal and has $p_g=0$. Moreover, since nodes are 
rational singularities, $T(S/\sigma)=T(X_\sigma)$.
\end{remark}

Using the fact that by the result of Bloch, Kas and Liebermann (cf. \cite{bkl}) Bloch's conjecture is true for surfaces $S$ with $\kod(S) \leq 1$, we obtain the following: 
\begin{corollary}[{\cite[Corollary 1.5]{BauerInoue}}]\label{corBl}
Let $S$ be a surface of general type with $p_g(S)=0$ and assume that 
$G=(\mathbb Z/ 2\mathbb Z)^2 \triangleleft \mathrm{Aut}(S)$. Assume that for each $\sigma \in G\setminus\{0\}$
 the quotient $S/\sigma$ has $\mathrm{kod}(S/\sigma)\leq 1$, then $S$ satisfies 
Bloch's conjecture.
\end{corollary}

%%%%%%%%%%%%%%%%%%%%%%%%%%%%%%%%%%%%%%%%%%%%%%%%%%%%%%%%%%%%%%%%%%%%%%%%%%%%%%%%%%%%%

\section{Involutions on surfaces of general type} 
In this section we collect some results regarding involutions on  surfaces of general type, 
that we  need in Section \ref{main}. We start fixing some notation.

Let $S$ be a minimal regular surface of general type with an involution  $\sigma$.
Then $\sigma$ is biregular and its fixed locus  is the union of $k$ isolated points $P_1\ldots,P_k$
and a smooth (not necessarily connected) curve $R$.
We denote by $p\colon S\rightarrow \Sigma:=S/\langle\sigma\rangle$ the projection onto the quotient,
by $B$ the image of $R$ and by $Q_j$ the image of $P_j$, $j=1,\ldots, k$.
The surface $\Sigma$ is normal, $Q_1,\ldots, Q_k$ are nodes and they are the only singularities of $\Sigma$.
Let $h\colon V \rightarrow S$ be the blow-up of $S$ at $P_1,\ldots, P_k$ and  $E_j$ be the exceptional 
curve over $P_j$, $j=1,\ldots, k$.\\
The involution $\sigma$ induces an involution $\tilde{\sigma}$ on $V$ whose fixed locus is the union of 
$R'=h^{-1}R$ and of $E_1,\ldots, E_k$.
Let $\pi\colon V\rightarrow W:=V/\langle\tilde{\sigma}\rangle$ be the projection onto the quotient
and set $B':=\pi(R')$, $ A_j:=\pi(E_j)$, $j=1,\ldots, k$.
The surface $W$ is smooth and the $A_j$ are disjoint $(-2)$-curves.
Let $g$  be the morphism induced by $h$, $g$ is the minimal resolution of the singularities of $\Sigma$ and 
 we have the following commutative diagram:
\begin{equation}
\xymatrix{
V \ar[r]^h \ar[d]_\pi&S \ar[d]^p\\
W\ar[r]_g& \Sigma
}
\end{equation}
 
\noindent Let  $ B'':=B'+\sum_{j=1}^k A_j$ be the branch divisor of $\pi$, one has 
$\pi_*\mathcal{O}_V= \mathcal{O}_W\oplus \mathcal{O}_W(-\Delta')$ with $B''\equiv 2\Delta'$.

Now let us assume that $W$ is of general type. We denote by $P$ the minimal model of  $W$, 
by $\rho\colon W\rightarrow P$  the corresponding projection and 
let $\overline{B}:= \rho_*(B'')$. Then $\overline{B}$ is an even divisor linearly equivalent to $2\Delta$, 
where $\Delta:=\rho_*(\Delta')$. $V$ is the canonical resolution of the double cover of $P$ 
branched along $\overline{B}$ and, by \cite[Lemma 6]{Horikawa}, one has:
\begin{equation}\label{eqH1}
K_S^2-k=K_V^2=2(K_P+\Delta)^2-2\sum_i (x_i-1)^2\,,
\end{equation}
\begin{equation}\label{eqH2}
\chi(\mathcal O_S)=\chi(\mathcal O_V)=2\chi(\mathcal O_P)+\frac{1}{2}(K_P+\Delta).\Delta-\frac{1}{2}\sum_ix_i (x_i-1)\,,
\end{equation} 
where $x_i:= \lfloor \frac{m_i}{2}\rfloor$, being $m_i$ the multiplicity of the singular point $b_i$ of $\overline{B}$.
We recall the following: 
\begin{proposition}[{\cite[Proposition 1]{Bom}}]\label{bom}
Let $S$ be a minimal surface of general type. Let $C$ be an irreducible curve on $S$, then
$K_S.C\geq 0$ and if $K_S.C= 0$, then $C^2=-2$ and $C$ is a rational non-singular curve.
\end{proposition}

We can now prove:

\begin{proposition}\label{NotGT}
Let $S$ be a (minimal) surface of general type with $K_S^2=6$, $p_g(S)=0$ and that contains no rational curves except
at most a $(-2)$-curve $L$. 
Let $\sigma$ be an involution on $S$ such that one of the following holds:
\begin{itemize}[leftmargin=.9cm]
\item[(i)] either $Fix(\sigma)$ contains more than 8 isolated points and a non-rational curve;
\item[(ii)] or  $Fix(\sigma)$ is given by 6 isolated  points and an elliptic curves $C$ 
such that $C\cap L=\emptyset$;	
\item[(iii)] or  $Fix(\sigma)$ is given by 4 isolated  points, an elliptic curves $C$  and the $(-2)$-curve $L$: $C\cap L=\emptyset$.
\end{itemize}
Then $\Sigma:=S/\langle\sigma\rangle$ is not of general type.
\end{proposition}
\begin{proof}
We use the notation of above and aiming for a contradiction we assume $W$ of general type.\\ 
Since $S$ is a minimal surface of general type with $K_S^2=6$ and $p_g(S)=0$, then 
$p_g(P)=q(P)=0$ and  from formulas (\ref{eqH1}), (\ref{eqH2}) we get 
\begin{equation}\label{eqH3}
5-\frac{k}{2}=K_P^2+K_P.\Delta+\sum_i(x_i-1)\,.
\end{equation}

 \noindent We claim that $K_P.\overline{B}= 2 K_P.\Delta>0$, indeed: 
  $\overline{B}=\sum_{l=1}^r B_l+ \sum_j \rho_*A_j$, where, by assumption,
 each $B_l$ is an irreducible curve,   $r>0$ and $B_t$ is non-rational for 
 at least one  $t \in \{1,\ldots, r\}$.
By Proposition \ref{bom},  $K_P.B_l\geq 0$ for any $l=1,\ldots, r$ and $K_P.B_t> 0$ and the claim follows.

\textit{Case (i)}: $k\geq 8$ and by (\ref{eqH3}) we get:
$$1\geq K_P^2+K_P.\Delta+\sum_i(x_i-1)\,,$$
but this is not possible because $K_P^2\geq 1$ and $K_P.\Delta>0$.

\textit{Case (ii)}: $k=6$ and by (\ref{eqH3}) we get:
$$2=K_P^2+K_P.\Delta+\sum_i(x_i-1)\,,$$
therefore $K^2_P=K_P.\Delta=1$ and $x_i=1$ for any $i$. 

\noindent Since $e(V)=12$, it holds $e(W)=6+\frac{1}{2} e(B'')$; by assumption $B'':=\sum_{j=1}^6 A_j+\tilde C$
with $\tilde C:=\pi(h^{-1}(C)) $ a smooth elliptic curve: $e(B'')=12$. By Noether's formula $K^2_W=0$, hence $\rho$ is the blow-down of exactly one
  $(-1)$-curve $E$. In particular $E$  does not intersect any $(-2)$-curves.\\
By assumption, $S$  contains no rational curves except at most a  $(-2)$-curve $L$, such that $L\cap C =\emptyset$.
If $S$ contains such a curve $L$, then $\sigma$ maps $L$ onto itself hence
 two $\sigma$-fixed points lie on it;  an easy computation shows that 
$\pi(\tilde L )$ is a $(-2)$-curve on $W$, being $\tilde L\subset V$ the strict transform of $L$.
Therefore, the rational curve $E$ must intersect $\tilde C$ in at least 4 points (by Hurwitz' formula)
and $\overline{B}$ contains a 
singular point with $m_i\geq 4$, i.e. $x_i\geq 2$.

\textit{Case (iii)}: $k=4$ and by (\ref{eqH3}) we get:
$$3=K_P^2+K_P.\Delta+\sum_i(x_i-1)\,,$$
therefore $K^2_P\leq 2$.		\\
Let $\tilde C:=\pi(h^{-1}(C)) $ and $\tilde L:=\pi(h^{-1}(L))$,
 then $B''= \sum_{j=1}^4 A_j+ \tilde C+\tilde L$  is the disjoint union of five rational curves
 and an elliptic curve: $e(B'')=10$. Since $e(V)=10$,  $e(W)=\frac{1}{2} e(V)+\frac{1}{2} e(B'')=10$ and 
 by Noether's formula $2=K^2_W\leq K^2_P $.
 We get that $W=P$ is minimal,  $K_P.\Delta=1$ and $x_i=1$ for any $i$. 
By  (\ref{eqH2}), it follows $\Delta^2=-3 $ and so $-12=B''^2=-12$.
It is direct to show that $\tilde L$ is a $(-4)$-curve, hence  $\tilde C^2-12=B''^2$. 
We have an elliptic curve $\tilde C$ with $\tilde C^2=0$ on a
 minimal surface of general type, it contradicts Proposition \ref{bom}.  
 \end{proof}

%%%%%%%%%%%%%%%%%%%%%%%%%%%%%%%%%%%%%%%%%%%%%%%%%%%%%%%%%%%%%%%%%%%%%%%%%%%%%%%%%%%%%
\section{Generalized Burniat type surfaces}\label{GBTS}
In this section we recall  the construction of generalized Burniat type surfaces. 
For further details we refer to \cite{BCF14}.

For $j= 1,2,3$, let $E_j=\mathbb C/\langle 1, \tau_j\rangle$ be an elliptic curve
and denote by $z_j$ a uniformizing parameter on $E_j$.
Let $\mathcal{L}_j$  be the Legendre $\mathcal{L}$-function for $E_j$: $\mathcal{L}_j$ is a meromorphic function on $E_j$
and $\mathcal{L}_j\colon E_j\rightarrow \mathbb{P}^1$ is a double cover branched over four distinct points:
$\pm 1, \pm a_j\in \mathbb{P}^1\setminus \{0, \infty\}$. 
It is well known that the following statements hold  (see \cite[Lemma 3-2]{inoue} and \cite[Section 1]{BC11burniat1}):
\begin{itemize}[leftmargin=.9cm]
\item $\mathcal{L}_j(0)=1$, $\mathcal{L}_j(\frac{1}{2})=-1$,  $\mathcal{L}_j(\frac{\tau_j}{2})=a_j$,
 $\mathcal{L}_j(\frac{\tau_j+1}{2})=-a_j$;
\item let $b_j:=\mathcal{L}_j(\frac{\tau_j}{4})$, then $b_j^2=a_j$;
\item $\frac{\mathrm d\mathcal L_j}{\mathrm d z_j}(z_j)=0$ if and only if 
$z_j\in\left\{0,\frac 12, \frac {\tau_j}2, \frac {\tau_j+1}2\right\}$
 (since these are the ramification points of $\mathcal{L}_j$);

\item $\mathcal{L}_j(z_j)=\mathcal{L}_j(z_j+1)=\mathcal{L}_j(z_j+\tau_j)=\mathcal{L}_j(-z_j)=-\mathcal{L}_j\left(z_j+\frac 12\right)$;

\item $\mathcal{L}_j\left(z_j+\dfrac{\tau_j}2\right)= \dfrac{a_j} {\mathcal L_j(z_j)}$.

\end{itemize}

 For $j	\in\{ 1,2,3\}$,  we define an action 
of $\{(\zeta_j,\eta_j,\epsilon_j)\}\cong (\mathbb{Z}/2 \mathbb{Z})^3$ on $E_j$,  as follows:
\begin{equation}\label{Z23-action}
\begin{array}{ccc}
(z_j\mapsto -z_j)&\hat{=}&(1,0,0)\\[3pt]
(z_j\mapsto -z_j+\frac{\tau_j}{2})&\hat{=}&(0,1,0)\\[3pt]
(z_j\mapsto -z_j+\frac{1}{2})&\hat{=}&(0,0,1),
\end{array}
\end{equation}
and we consider the induced action of $\mathcal{G} =
\{(\zeta_1,\eta_1,\epsilon_1,\zeta_2,\eta_2,\epsilon_2,\zeta_3,\eta_3,\epsilon_3)\}$
on $T:=E_1\times E_2\times E_3$. 
 We define also the following map:
\begin{equation}
\label{diag1}
\begin{array}{ccc}
\pi'\colon E_1\times E_2\times E_3 &\longrightarrow & \mathbb P^1\times\mathbb P^1\times\mathbb P^1  \\[6pt]
\phantom{\pi'}(z_1,z_2,z_3) &\longmapsto & \left(\dfrac{\mathcal L_1(z_1)}{b_1}, \dfrac{\mathcal L_2(z_2)}{b_2},\dfrac{\mathcal L_3(z_3)}{b_3}\right).
\end{array}
\end{equation}

\noindent The $\mathcal G$-action on $T$ induces, via $\pi'$, an action of
$ \mathcal H:=(\mathcal H_1)^3\cong (\mathbb Z/ 2\mathbb Z)^6$ on 
$P_1:=(\mathbb{P}^1)^3$, where 
 $\mathcal H_1\cong (\mathbb Z/ 2\mathbb Z)^2 $ acts on $\mathbb{P}^1$ in this way:
\begin{equation}\label{Z22-action}
\begin{array}{ccl}
(1,0)&\hat{=}&((s:t)\mapsto (t:s)),\\
(0,1)&\hat{=}&((s:t)\mapsto (s:-t)),
\end{array}
\end{equation}
being $(s:t)$  homogeneous coordinates of $\mathbb{P}^1$. 

\noindent Let $Y\subset \mathbb P^1\times\mathbb P^1\times\mathbb P^1$ be an irreducible Del Pezzo surface 
of degree 6 invariant under a subgroup $H \cong (\mathbb Z/ 2\mathbb Z)^2\ \triangleleft\mathcal H$.\\
The inverse image $\hat{X}:=\pi'^{-1}(Y)$ of $Y$ under $\pi'$ is an irreducible hypersurface in the product of three smooth elliptic curves 
$T:= E_1 \times E_2 \times E_3$, which is of multi degree $(2,2,2)$.
\begin{definition}
 $\hat X$  is called a  \textit{Burniat hypersurface in $T$}.
\end{definition}

\begin{remark}\label{B-hyp}
According to \cite{BCF14}, 
every Burniat hypersurface is given by one of the following equations: 
\begin{equation}\label{righteq1}
\begin{array}{r}
\hat X_\nu= \{(z_1,z_2,z_3) \in T \mid\nu_1(\mathcal L_1(z_1)\mathcal L_2(z_2)\mathcal L_3(z_3) +  b_1b_2b_3)+
 \\[6pt]
\nu_2(\mathcal L_1(z_1) b_2b_3 +  b_1\mathcal L_2(z_2)\mathcal L_3(z_3))=0 \}	\,,
\end{array}
\end{equation}
\begin{equation}\label{righteq2}
\hat X_{\mu}= \{(z_1,z_2,z_3) \in T \mid \mathcal L_1(z_1)\mathcal L_2(z_2)\mathcal L_3(z_3)=\mu \}\,,
\end{equation}
\begin{equation}\label{righteq3}
\hat X_b= \{(z_1,z_2,z_3) \in T \mid \mathcal L_1(z_1)\mathcal L_2(z_2)\mathcal L_3(z_3)=b_1b_2b_3\}\,,
\end{equation}

\noindent where ${\nu}:=(\nu_1:\nu_2) \in \mathbb P^1$, $\mu\in \mathbb C$ and $b:=b_1b_2b_3$. 

\noindent Recall  that we are
considering only values of ${\nu}$ (resp.  $\mu$) such that $\hat X_\nu$ (resp. $\hat X_\mu$) is irreducible, i.e.
$(\nu_1/\nu_2) \neq \pm 1 $ and $\mu \neq 0$.
\end{remark}

\begin{remark}
By construction, a Burniat hypersurface  $\hat X$ has at most finitely many nodes as singularities. Therefore, denoting by $\epsilon\colon X' \rightarrow \hat X$ the minimal resolution of its singularities, we have 
that  $K_{X'}=\epsilon^* K_{\hat X}$ and 
 $X'$ is a minimal surface of general type with $K^2_{X'}=48$ and $\chi(X')=8$.
\end{remark}

Let $\mathcal G_0\cong (\mathbb Z/ 2\mathbb Z)^6\triangleleft\, \mathcal G\cong
(\mathbb Z/ 2\mathbb Z)^3 \times(\mathbb Z/ 2\mathbb Z)^3
\times (\mathbb Z/ 2\mathbb Z)^3$ be the group: 
$$\mathcal G_0:=\{(\zeta_1,\eta_1,\epsilon_1,\zeta_2,\eta_2,\epsilon_2,\zeta_3,\eta_3,\epsilon_3) 
\mid\eta_1=\eta_2=\eta_3,\, \epsilon_1+\epsilon_2+\epsilon_3=0\}\,.$$

\noindent Then we have the following:

\begin{lemma}[{\cite{BCF14}}] 
\begin{enumerate}
\item   $\hat X_\nu$ is invariant under the group 
$$\mathcal G'_1:=\{(\zeta_1,\eta_1,0,\zeta_2,\eta_1,\epsilon_2,\zeta_3,\eta_1,\epsilon_3)  \mid \epsilon_2+\epsilon_3=0\}\cong (\mathbb Z/ 2\mathbb Z)^5\triangleleft \mathcal G_0\,.$$
\item  $\hat X_\mu$ is invariant under the group
$$\mathcal G_1:=\{(\zeta_1,0,\epsilon_1,\zeta_2,0,\epsilon_2,\zeta_3,0,\epsilon_3)   \mid
 \epsilon_1+\epsilon_2+\epsilon_3=0\}\cong (\mathbb Z/ 2\mathbb Z)^5\triangleleft \mathcal G_0\,.$$
\item $\hat X_b$ is invariant under $\mathcal G_0$. % $\hat X_1$
\end{enumerate}
\end{lemma}

\begin{remark} 

Let $g\in \mathcal G_0 \setminus \{0\}$ be an element fixing points on $ T$.
By \cite[Proposition 4.3]{BC13}, $g$ is then an element in Table \ref{FixEl}. 

\begin{table}[!h]
\caption{The element of $\mathcal G_0$ having fixed points on $T$}
\label{FixEl}
\begin{tabular}{p{.3cm}||p{.3cm}p{.3cm}p{.3cm}||p{.3cm}p{.3cm}p{.3cm}p{.3cm}p{.3cm}p{.3cm}
||p{.3cm}p{.3cm}p{.3cm}p{.3cm}p{.3cm}p{.3cm}p{.3cm}p{.4cm}}
&$g_1$&$g_2$&$g_3$&$g_4$&$g_5$&$g_6$&$g_7$&$g_8$&$g_9$&
$g_{10}$&$g_{11}$&$g_{12}$&$g_{13}$&$g_{14}$&$g_{15}$&$g_{16}$&$g_{17}$\\
\hline 
$\zeta_1$ &  $ 0 $ & $ 0 $& $ 1 $&$ 0 $ &$ 1 $& $ 1 $&$ 0 $ & $ 0 $ & $ 0 $&$ 1 $ &$ 1 $&$ 0 $ &  $0$& 
$0$ & $ 0 $& $ 1 $& $ 1 $ 	\\
$\eta_1 $&      	$ 0 $ & $ 0 $& $ 0 $& $ 0 $& $ 0 $&$ 0 $ & $ 0 $&$ 0 $ & $ 0 $&$ 0 $ &$ 0 $&$ 0 $ & $ 0 $&
$ 1 $ & $ 1 $& $ 1 $&$ 1 $\\
$\epsilon_1 $&  $ 0 $ & $ 0 $& $ 0 $& $ 0 $& $ 0 $&$ 0 $ & $ 0 $&$ 1 $ & $ 1 $&$ 0 $ &$ 0 $&$ 1 $ & $ 1 $&
$ 0 $ & $ 0 $& $ 1 $&$ 1 $\\
\hline
$\zeta_2$&        	$ 0 $ & $ 1 $& $ 0 $& $ 1 $& $ 0 $& $ 1 $& $ 0 $&$ 0 $ & $ 0 $&$ 1 $ &$ 0 $&$ 1 $ & $ 0 $&
$ 0 $ & $ 1 $& $ 0 $&$ 1 $\\
$\eta_2 $&      	$ 0 $ & $ 0 $& $ 0 $& $ 0 $& $ 0 $&$ 0 $ & $ 0 $&$ 0 $ & $ 0 $&$ 0 $ &$ 0 $&$ 0 $ & $ 0 $&
$ 1 $ & $ 1 $& $ 1 $&$ 1 $\\
$\epsilon_2$& 	$ 0 $ & $ 0 $& $ 0 $& $ 0 $& $ 0 $& $ 0 $& $ 1 $&$ 0 $ & $ 1 $&$ 0 $ &$ 1 $& $ 0 $& $ 1 $&
$ 0 $ & $ 1 $& $ 0 $&$ 1 $\\
\hline
$\zeta_3$& 			$ 1 $ & $ 0 $& $ 0 $& $ 1 $& $ 1 $& $ 0 $& $ 0 $&$ 0 $ & $ 0 $& $ 1 $&$ 0 $& $ 0 $& $ 1 $&
$ 0 $ & $ 1 $& $ 1 $&$ 0 $\\
$\eta_3 $&      	$ 0 $ & $ 0 $& $ 0 $& $ 0 $& $ 0 $&$ 0 $ & $ 0 $&$ 0 $ & $ 0 $&$ 0 $ &$ 0 $&$ 0 $ & $ 0 $&
$ 1 $ & $ 1 $& $ 1 $&$ 1 $\\
$\epsilon_3$&	$ 0 $ & $ 0 $& $ 0 $& $ 0 $& $ 0 $& $ 0 $& $ 1 $&$ 1 $ & $ 0 $&$ 0 $ &$ 1 $& $ 1 $& $ 0 $&
$ 0 $ & $ 1 $& $ 1 $&$ 0 $\\
\end{tabular}
\end{table}
1) Let $\hat X:=\hat X_b$. In Table \ref{FixEl}, the elements $g_1$-$g_3$ fix pointwise a surface $S\subset T$.
Each element $g_4$-$g_9$ fixes pointwise a curve $C\subset T$ and its fixed locus has non trivial 
intersection with $\hat X$ since $\hat X\subset T$ is an ample divisor.
Finally, the elements $g_{10}$-$g_{17}$  have isolated fixed points on $T$; 
in particular, the elements $g_{11}$-$g_{17}$ have fixed points on $\hat X$,
 while the fixed locus of element $g_{10}$
intersects $\hat X$ only for special choices of the three elliptic curves.

2) The same holds for $\hat X_{\nu}:=\pi'^{-1}(Y_{\nu})$ (resp. $\pi'^{-1}(Y_\mu)$), 
considering only the elements  $g_{1}$-$g_{7}$,$g_{10}$,$g_{11}$,$g_{14}$,$g_{15}$ (resp. $g_{1}$-$g_{13}$),
i.e. the ones belonging to $\mathcal G'_1$ (resp. $\mathcal G_1$).
In particular, the fixed locus of  element $g_{10}$ intersects $\hat X$ only for special choices of the three elliptic curves and the parameter $\nu$ (resp. $\mu$).
\end{remark}

\begin{definition}
Let $\hat X$ be a  Burniat hypersurface in $E_1 \times E_2 \times E_3$ and let  $G\cong  (\mathbb Z/ 2\mathbb Z)^3$ be 
a subgroup of $\mathcal G_0$ 
 acting  freely on $\hat X$.
The minimal resolution $S$ of the quotient surface $\hat S:= \hat X/G$ is called a 
\textit{generalized Burniat  type (GBT) surface}. We call $\hat S$ the \textit{quotient model of $S$}. 
\end{definition}

\begin{remark}1)  Since $G$ acts freely and $\hat X$ has at most nodes as singularities, 
$\hat S$ is singular if and only if $\hat X$ is singular and  $\hat S$ has at most nodes as singularities.

2) A generalized Burniat type surface $S$ is a smooth minimal surface of general type with
$K^2_S=6$ and $\chi(S)=1$.
\end{remark}

In \cite{BCF14}, GBT surfaces have been completely classified. In particular, it has been shown there are exactly four families of GBT surfaces with 
$p_g=q=0$:

 \begin{theorem}\label{clthm}

Let $S\rightarrow \hat S= \hat X /G$ be a regular generalized Burniat type surface $S$ then
 $(\hat{X},G) \in \{(\hat X_\nu,G_1),\, (\hat X_\mu,G_2),\,(\hat X_b,G_j), \,j=3,4\}$, where
 the groups $G_1,\,G_2,\,G_3,\,G_4$ are  in Table \ref{q0}.
 \begin{table}[!h]
\begin{tabular}{c|ccc|ccc|ccc|c}
 & $\zeta_1$ & $\eta_1$& $\epsilon_1$  & $\zeta_2$ & $\eta_2$& $\epsilon_2$ & $\zeta_3$ & $\eta_3$& $\epsilon_3$&\quad\\
 \hline
 \multirow{3}{*}{$G_1$}&1&0&0&1&0&0&1&0&0   \\
&0&1&0&1&1&0&1&1&0\\
&0&0&0&0&0&1&1&0&1\\
\hline
 \multirow{3}{*}{$G_2$}&1&0&0&0&0&1&1&0&1\\
&0&0&1&0&0&0&1&0&1\\
&0&0&0&1&0&1&0&0&1\\

	\hline
\multirow{3}{*}{$G_3$}&1&0&0&0&0&1&1&0&1\\
&0&1&0&0&1&0&1&1&0\\
&0&0&1&1&0&1&1&0&0\\
\hline
\multirow{3}{*}{$G_4$}&1&0&1&0&0&1&1&0&0\\
&0&1&0&0&1&0&1&1&0\\
&0&0&0&1&0&1&1&0&1\\
\hline
&&&&&&&&&\\	
\end{tabular}\caption{The groups $G_j$}\label{q0}
\end{table}

\end{theorem}

To fix the notation, let us call a surface $S$ a {\it generalized Burniat type (GBT) surface of type $j$} if $S$ belongs to the (uniquely determined) family number $j$ in Tables \ref{q0}.

\begin{remark}
As shown in \cite{BCF14} GBT surfaces of type $j$ ($1 \leq j \leq 4$) have pairwise non isomorphic fundamental groups. In particular, they belong to different connected components of the moduli space of surfaces of general type.
\end{remark}

\begin{remark}\label{eqGBT}
Let $\hat S:= \hat X/G$ be the quotient model of a regular GBT surface $S$.
According to Theorem \ref{clthm}, there are the following possibilities:
\begin{itemize}[leftmargin=.9cm]
\item[a)]
$\hat X=\hat X_\nu:= \{(z_1,z_2,z_3) \in T \mid \nu_1L_1(z_1,z_2,z_3) + \nu_2L_2(z_1,z_2,z_3)=0 \}$ and $G=G_1$,
where  $\nu:=(\nu_1:\nu_2) \in\mathbb P^1 $, $(\nu_1/\nu_2)\neq \pm 1$,
\begin{eqnarray*}
L_1(z_1,z_2,z_3)&:=&\mathcal L_1(z_1)\mathcal L_2(z_2)\mathcal L_3(z_3) +b_1b_2b_3 \quad \mbox{ and }\\
L_2(z_1,z_2,z_3)&:=&\mathcal L_1(z_1) b_2b_3 + b_1\mathcal L_2(z_2)\mathcal L_3(z_3)\,.
\end{eqnarray*}

\noindent Note that $g_0:=(1,0,0,1,0,0,1,0,0)\in G_1$, and its fixed locus $F_0=\{(z_1,z_2,z_3) \in T \mid 2z_1=2z_2=2z_3=0\}$ intersects $\hat X_\nu$ if 
$$\qquad \nu \in B_0:=\{ ( L_2(z_1,z_2,z_3): -L_1(z_1,z_2,z_3) ) \,,  (z_1,z_2,z_3)\in F_0\}\,.$$

\noindent In other words, if $\nu\in B_0$ then $G_1$ does not act freely on $\hat X_\nu$ and 
therefore does not give rise to a GBT surface.

\item[b)] $\hat X=\hat X_\mu:= \{(z_1,z_2,z_3) \in T \mid \mathcal L_1(z_1)\mathcal L_2(z_2)\mathcal L_3(z_3)=
\mu \}$, with $ \mu \in \mathbb C, \mu \neq 0$ and $G=G_2$.  
Since  $g_0\in G_2$, if $\mu \in B':=  \{\mathcal L_1(z_1)\mathcal L_2(z_2)\mathcal L_3(z_3),  (z_1,z_2,z_3)\in F_0 \}$,
then $G_2$  does not act freely on $\hat X_\mu$; moreover  (see e.g. \cite{inoue}):
$$B'=\{\pm 1, \pm a_i, \pm a_i a_j, \pm a_1a_2a_3\} \,\mbox{ with }\, i\neq j\in\{1,2,3\}\,.$$
 We remark that this case gives rise to  the family of {\it primary  Burniat surfaces} (see \cite{BC11burniat1, BC13}).

\item[c)] 
 $\hat X=\hat X_{b}:=  \{(z_1,z_2,z_3) \in T \mid \mathcal L_1(z_1)\mathcal L_2(z_2)
 \mathcal L_3(z_3)=b\}$, with  $b:=b_1b_2b_3$ and $G=G_j$,  $j=3,4$.
\end{itemize}
\end{remark}

We already remarked that  $\hat X$ has at most finitely many nodes as singularities. 
The next statement shows  that  either $\hat{X}$ is smooth or has exactly  eight nodes.

\begin{proposition}\label{NumberOfSing}
Let $S\rightarrow \hat S= \hat X /G$ be a regular generalized Burniat type surface, i.e.,  
 $(\hat{X},G) \in \{(\hat X_\nu,G_1),\, (\hat X_\mu,G_2),\,(\hat X_b,G_j), \,j=3,4\}$.
Then:
\begin{itemize}[leftmargin=.9cm]
\item[1)] $\hat X_\nu$ is singular if and only if $\nu\in B:=\{(\pm b_1:1), (1:\pm b_1) \}$ and
$\Sing(\hat X_\nu)=\{(z_1, \pm\frac{1}{4}, \pm\frac{1}{4}) \mid 2z_1=0,  \nu b_1+\mathcal{L}_1(z_1)=0\}
\cup \{(z_1, \frac{\tau_2}{2}\pm\frac{1}{4}, \frac{\tau_3}{2}\pm\frac{1}{4}) \mid  2z_1=0,  b_1+\nu\mathcal{L}_1(z_1)=0\}$;
\item[2)] $\hat X_\mu$ is smooth; 
\item[3-4)] $\hat X_b$ is singular if and only if $b:=b_1b_2b_3 \in B'$ and 
$\Sing(\hat X_b)=\{(z_1,z_2,z_3)\in \hat X_b \mid 2z_1=2z_2=2z_3=0\}$.
\end{itemize}
In particular, either $\hat X$ is smooth or its singular locus consists of exactly 8 nodes.
\end{proposition}

\begin{proof}
We start with cases 2) and 3-4). 
\noindent

Let $f(z_1,z_2,z_3) := \mathcal L_1(z_1)\mathcal L_2(z_2) \mathcal L_3(z_3)$.  
It is easy to see that 
$$\Sing(\hat X_{\mu})=\{z\in \hat X_{\mu} \mid \nabla f(z)=0\} =
		 \{(z_1,z_2,z_3) \in \hat  X_{\mu} \mid 2z_1=2z_2=2z_3=0\} $$
since $\dfrac{d f}{d z_i}= \dfrac{d\mathcal L_i}{d z_i} \mathcal L_{i+1} \mathcal L_{i+2}$
 (the indices $i \in \{1,2,3\}$ have to be considered mod 3).
 We observe that  $\Sing(\hat X_{\mu})=\hat X_{\mu}\cap F_0$, being $ F_0$ the fixed locus 
of $g_0:=(1,0,0,1,0,0,1,0,0)$.

Since  $g_0 \in G_2$,  we get that the surfaces of case 2) are smooth. 

In case 3-4), $\mu:=b=b_1b_2b_3$ and  $\hat{X}_b$ is singular if and only if $b \in B'=\{f(z)\mid z \in F_0\}$; 
we have to determine the number of nodes.

\noindent We note that   if $(z_1,z_2,z_3)\in F_0$ then
$f(z_1,z_2,z_3)=f(z_1+\frac12,z_2+\frac12,z_3)=f(z_1,z_2+\frac12,z_3+\frac12)=f(z_1+\frac12,z_2,z_3+\frac12)$,
and all these points are in $F_0$.
Since $a_j=b_j^2\neq \pm 1$ and $\hat X_b$ is invariant under $(*,1,*,*,1,*,*,1,*)$,
it is easy to show that $b=\pm 1$ iff $b= \pm a_1a_2a_3 $ and 
that $b = \pm a_i $ iff $b= \pm a_{i+1}a_{i+2} $ (indices have to be considered mod 3);
in particular, $|B'|=8$.
\noindent  Since $|F_0|=64$, it follows that for any choice of $b \in B'$ exactly 8 points of $F_0$ belong to $\hat X_b$, 
i.e. $\Sing(\hat X_b)$ consist of exactly 8 nodes. 

\

In case 1), let $f_\nu(z_1,z_2,z_3):=\nu_1L_1(z_1,z_2,z_3) + \nu_2L_2(z_1,z_2,z_3)$.
We note that if $\nu_1=0$ or $\nu_2=0$, arguing as above  we  get that
$\hat X_\nu$ is smooth ($g_0\in G_1$), so we may assume $\nu_2 =1$ and 
$\nu:=\nu_1\in \mathbb C\setminus \{0\}$.
We recall that  $\mathcal{L}_i$ has poles in $z_i=\frac{\tau_i}{2}\pm \frac14$ and zeroes in $z_i=\pm \frac14$.

\noindent We start considering charts such that $z_i\neq\frac{\tau_i}{2}\pm \frac14$ for $i=1,2,3$.
It is easy to see that 
\begin{equation}\label{syst}
\nabla f_\nu=0 \Longleftrightarrow
\left\{\begin{array}{c}
\mathcal L'_1(\nu\mathcal L_2\mathcal L_3+b_2b_3)=0\\
\mathcal L'_2\mathcal L_3(\nu\mathcal L_1+b_1)=0\\
\mathcal L_2\mathcal L'_3(\nu\mathcal L_1+b_1)=0
\end{array}\right.
\mbox{ with }\quad \mathcal L'_i =\dfrac{d\mathcal L_i}{d z_i}\,.
\end{equation}
If $\nu\mathcal L_1+b_1=0$ for a point in $\hat X$, then 
$$f_\nu = \mathcal L_2\mathcal L_3 (\nu\mathcal L_1+b_1)+
b_2b_3(\nu b_1 +\mathcal L_1)=\nu b_1 +\mathcal L_1=0$$
 hence $\nu=\pm 1$, i.e. $\hat X$ is not irreducible,
a contradiction (see Remark \ref{B-hyp});
analogously we can assume $\nu\mathcal L_2\mathcal L_3+b_2b_3\neq 0$.

\noindent Since $\mathcal L_i$ and $\mathcal L'_i$ have no common zeroes, 
$\mathcal L'_2=0$ if and only if $\mathcal L'_3=0$; in this case the solutions of
$\nabla f_\nu=0$ are
points in $F_0$, the fixed locus of $g_0\in G_1$.
Therefore $(z_1,z_2,z_3)\in \Sing(\hat X)$  if and only if it satisfies the following equations:
 $\mathcal L_2(z_2)=\mathcal L_3(z_3)=0$, $\mathcal L'_1(z_1)=0$ and    $f_\nu=\nu b_1+\mathcal L_1(z_1)=0$.
 It is immediate to see that the last two equations have common solutions 
if and only if  $\nu\in B=\{\pm b_1, \pm b_1^{-1}\}$;
if $\nu \in B$, we find 4 nodes:
$z_1=\mathcal{L}_1^{-1}(-\nu b_1)$, $z_2\in \{\pm \frac{1}{4}\}$ and $z_3\in \{\pm \frac{1}{4}\}$.

We now consider charts such that  $z_2\neq\pm \frac14$ and $z_3\neq\pm \frac14$
 then the affine equation $f_\nu=0$ can be written as follows:
$$f_\nu=\overline{\mathcal L}_2 \overline{\mathcal L}_3 b_2b_3(\nu b_1+\mathcal L_1)+ (\nu \mathcal L_1+b_1)$$
being $\overline{\mathcal L}_i:=\mathcal L_i^{-1}$, $i=2,3$.
 Arguing as above, one gets that 
 $(z_1,z_2,z_3)\in \Sing(\hat X)$ if and only if it satisfies the following equations:
 $\overline{\mathcal L}_2(z_2)=\overline{\mathcal L}_3(z_3)=0$, $\mathcal L'_1(z_1)=0$ and 
 $f_\nu=\nu \mathcal L_1(z_1) +b_1=0$.
 The last two equations have common solutions if and only if
$\nu\in B=\{\pm b_1, \pm b_1^{-1}\}$;
if $\nu \in B$, we find  other 4 nodes, namely:
 $z_1=\mathcal{L}_1^{-1}(- \frac{b_1}{\nu})$, $z_2\in \{\frac{\tau_2}{2}\pm \frac{1}{4}\}$ and
 $z_3\in \{\frac{\tau_3}{2}\pm \frac{1}{4}\}$.

Considering the other  charts, one finds either no singular points, 
or four of the eight nodes we  found. 

\end{proof}

\begin{corollary}
Let $S\rightarrow \hat S:= \hat X/G$ be a generalized Burniat surface.
Then either $\hat S$ is smooth or its singular locus is given by exactly one node. 
\end{corollary}

\begin{proof}
We simply note that  $G\cong (\mathbb Z/ 2\mathbb Z)^3$ acts freely on $\hat X$;
in particular it acts transitively on the set of nodes.
\end{proof}

%%%%%%%%%%%%%%%%%%%%%%%%%%%%%%%%%%%%%%%%%%%%%%%%%%%%%%%%%%%%%%%%%%%%%%%%%%%%%%%%%%%%%
\section{the main result}\label{main}

In this section we give a prove (using Corollary \ref{corBl} and Proposition \ref{NotGT}) that Bloch's conjecture holds for 
regular  generalized Burniat type surfaces. 

\begin{remark} Let $S\rightarrow \hat S:=\hat X/G$ be a GBT and let $\gamma \colon \hat X \rightarrow \hat S:=\hat X/G$ be the projection onto the quotient.
Let $\sigma$ be an involution on $\hat X$, % \in  \mathcal G_0 \setminus G
it defines an involution $\overline{\sigma} $ on 
$\hat S$: $\overline\sigma(\gamma(x)):=\gamma(\sigma(x))$
and $$Fix(\overline{\sigma} )= \bigcup_{g \in G} \gamma(Fix_{\hat X}( \sigma g))\,,$$
being $Fix_{\hat X}( \sigma ):=Fix( \sigma )\cap\hat X$.
Moreover,  $\overline{\sigma} $ lifts to an involution $\sigma':= \epsilon^{-1} \circ\overline{\sigma} \circ \epsilon$ on $S$.
\end{remark}

Generalized Burniat type surfaces  are constructed considering  
$G\cong  (\mathbb Z/ 2\mathbb Z)^3 \triangleleft \mathcal G_0$ acting freely on a Burniat hypersurface
$\hat X \subset T$, hence it is natural to consider involutions in  $\mathcal G_0 \setminus G$.
We start determining the fixed locus of elements in $\mathcal G_0$.

\begin{lemma}\label{FixLoc1}
Let $S\rightarrow \hat S= \hat X /G$ be a regular generalized Burniat type surface with
$(\hat X,G)=(\hat X_\nu, G_1)$.
Let $g \in \mathcal G'_1$ be an element fixing point on $X_\nu$, 
then its fixed locus $Fix(g)$ on $\hat X_\nu$  is as in Table \ref{EqFixLoc1}.
\begin{table}[!ht]
    \centering
\begin{tabular}{c|c|c}
&$Fix(g_i)$, $X_\nu$ smooth &$Fix(g_i)$, $X_\nu$ singular\\
\hline
$g_1:=(0,0,0,0,0,0,1,0,0)$&&\multirow{2}{*}{4 genus 5 curves} \\
$g_2:=(0,0,0,1,0,0,0,0,0)$&&\\
\cline{3-3}
$g_3:=(1,0,0,0,0,0,0,0,0)$&\multirow{-3}{*}{4 genus 5 curves} &{$\Gamma$}\\ %8 ell. c. pw. int. in a nd
\hline
$g_4:=(0,0,0,1,0,0,1,0,0)$&\multirow{3}{*}{32 pt} &\multirow{3}{*}{32 pt}  \\
$g_5:=(1,0,0,0,0,0,1,0,0)$&&\\
$g_6:=(1,0,0,1,0,0,0,0,0)$&&\\
\hline
$g_7:=(0,0,0,0,0,1,0,0,1)$&\multirow{1}{*}{16 pt, 8 ell. curves} &\multirow{1}{*}{8 nodes, 8 ell. curves}  \\
\hline
$g_{11}:=(1,0,0,0,0,1,0,0,1)$&\multirow{3}{*}{32  pt}&32 pt, 8 nodes\\
\cline{3-3}
$g_{14}:=(0,1,0,0,1,0,0,1,0)$&  &\multirow{2}{*}{32  pt}\\
$g_{15}:=(0,1,0,1,1,1,1,1,1)$& &\\
\end{tabular}
\caption{}\label{EqFixLoc1}
\end{table}

 \noindent  In Table \ref{EqFixLoc1}, $\Gamma$ denotes
 the disjoint union $\Gamma:=C_1 \sqcup C_2 \sqcup D$,
being $C_i$ ($i=1,2$) a genus 5 curve and $D$ the union of  8 elliptic curves each one passing through exactly 
two nodes and   such that a point belongs to two of them if and only if it is a node.

\end{lemma}

\begin{proof} By Proposition \ref{NumberOfSing}, $\hat X_\nu$ is singular if and only if 
 $\nu\in B:=\{(\pm b_1:1), (1:\pm b_1) \}$;
in this case the eight nodes on $\hat X_\nu$ are fixed by $g_{11}:=(1,0,0,0,0,1,0,0,1)$.

\begin{itemize}[leftmargin=.5cm]
\item[\underline{$g_1$}:]  since $Fix(g_1)\cap Fix (g_{11})= \emptyset$, the fixed locus of $g_1$ is independent from $\nu$. 
 It fixes the points $(z_1, z_2,  \overline{z_3})$ with $2\overline{z_3}=0$:
$ \overline{z_3}\in\big\{0,\frac{1}{2}, \frac{\tau_3}{2},\frac{\tau_3+1}{2}\big\}$, hence it cuts on $\hat X_\nu$ four disjoint curves
of genus 5: each one is given by an equation of multidegree (2,2) in $E_1\times E_2$.

\item[\underline{$g_2$}:] this case is analogous to the previous one.

\item[\underline{$g_3$}:] 
 it fixes the points  with $2z_1=0$ that is
$ z_1\in V:=\big\{0,\frac{1}{2}, \frac{\tau_1}{2},\frac{\tau_1+1}{2}\big\}$.

\noindent If $\hat X_\nu$ is smooth, this case is analogous to $g_1$:
we get  four disjoint smooth curves of genus 5 on $\hat X_\nu$.

\noindent If $\hat X_\nu$ is singular ($\nu\in B$), we rewrite the equation of $\hat X_\nu$ as follows:
$$\mathcal L_2(z_2) \mathcal L_3(z_3) (\nu \mathcal L_1(z_1)+b_1)+ b_2b_3(\nu b_1+ \mathcal L_1(z_1))=0\,.$$
For a fixed value $\nu \in  B $, there exists a unique $\overline{z_1}\in V$ such that $\nu b_1+\mathcal L_1(\overline{z_1})=0$, and the equation of $\hat X_\nu$ is satisfied if and only if $\mathcal L_2(z_2) =0$
 or $\mathcal L_3(z_3) =0$.
We get  four elliptic curves on $\hat X_\nu$ fixed by $g_{3}$:  $(\overline{z_1}, \pm\frac14, z_3)$, $(\overline{z_1}, z_2, \pm\frac14)$.
Analogously, considering  the element $z_1':=\overline{z_1}+\frac{\tau_1}2 \in V$   we get other four elliptic curves 
 $(z_1', \frac{\tau_1}2\pm\frac14, z_3)$ and 
 $(z_1', z_2, \frac{\tau_1}2\pm\frac14)$ on $\hat X_\nu$ fixed by $g_3$.
 We observe that that a point belongs to two of these eight  curves if and only if it is a node.\\
Considering $z_1\in V\setminus\{\overline{z_1},z_1' \}$, we get two disjoint curves of 
genus 5: both given by an equation of multidegree (2,2) in $E_2\times E_3$.

\item[\underline{$g_4$}:] 
 since $Fix(g_4)\cap Fix (g_{11})= \emptyset$, the fixed locus of $g_4$ is independent from $\nu$.
It fixes the points $(z_1,\overline{z_2}, \overline{z_3})$  with $2 \overline{z_2}=2 \overline{z_3}= 0$ and $2z_1\neq 0$,
since  $g_0:=(1,0,0,1,0,0,1,0,0) \in G_1$ has no fixed points:  for any pair $( \overline{z_2}, \overline{z_3})$, the equation
defining $\hat X_\nu$ has two distinct solutions,  whence $g_4$ fixes 32 points on $\hat X_\nu$.

\item[\underline{$g_5$-$g_6$}:] these cases are analogous to $g_4$.

\item[\underline{$g_7$}:] 

The involution $g_7$ fixes the points $(z_1,\overline{z_2}, \overline{z_3})$  with $2 \overline{z_2}=\frac12$, $2\overline{z_3}= \frac12$.
Let $(\mathcal{L}_i(z_i)_0:\mathcal{L}_i(z_i)_1) $ be the homogeneous coordinates of the point $\mathcal{L}_i(z_i)$. 
The equation of $\hat X_\nu$ is then 
$$\qquad \quad\begin{array}{c}
\nu[\mathcal L_1(z_1)_0\mathcal L_2(z_2)_0\mathcal L_3(z_3)_0 + 
 b_1b_2b_3\mathcal L_1(z_1)_1\mathcal L_2(z_2)_1\mathcal L_3(z_3)_1]+ \\[6pt]
[b_2b_3\mathcal L_1(z_1)_0\mathcal L_2(z_2)_1\mathcal L_3(z_3)_1 +
  b_1\mathcal L_2(z_2)_0\mathcal L_3(z_3)_0\mathcal L_1(z_1)_1]=0 \,.\end{array}$$
It follows easily from the properties of the  Legendre $\mathcal L$-function that 
$ \left(\mathcal L_i\left(\frac{1}{4}\right)_0: \mathcal L_i\left(\frac{1}{4}\right)_1\right)=(0:1)\,, 
\left(\mathcal L_i\left(\frac{1}{4}+\frac{\tau_i}{2}\right)_0: \mathcal L_i\left(\frac{1}{4}+\frac{\tau_i}{2}\right)_1\right)=(1:0)\,.
$
If $z_2=\pm \frac14$ and $z_3=\pm \frac14+\frac{\tau_3}2$ or
 $z_2=\pm \frac14+\frac{\tau_2}2$ and $z_3=\pm \frac14$, then the equation is satisfied for 
any $z_1 \in E_1$, i.e $g_7$ fixes 8 disjoint elliptic curves contained in the smooth locus of $\hat X_\nu$. 

\noindent If $z_2=\pm \frac14$ and $z_3=\pm \frac14$ then the equation becomes
$\nu b_1\mathcal{L}_1(z_1)_1+\mathcal{L}_1(z_1)_0=0$ that has two solutions if 
$\nu  \notin B$ (i.e. $\hat X_\nu$ is smooth) and one solution if $\nu  \in B$;
in other words,  if $\hat X_\nu$ is smooth $g_4$ fixes 8 isolated points, else $g_7$ fixes 4 nodes.
Analogously,  if $z_2=\pm \frac14+\frac{\tau_2}2$ and $z_3=\pm \frac14+\frac{\tau_3}2$ and
 if $\hat X_\nu$ is smooth $g_7$ fixes  other 8 isolated points, else it fixes  the other 4 nodes.

\item[\underline{$g_{11}$}:]  
we observe that  $Fix(g_{11})= Fix(g_7) \cap \{2z_1=0\}$ and arguing as above we distinguish three cases:
if $z_2=\pm \frac14$ and $z_3=\pm \frac14+\frac{\tau_3}2$ or
 $z_2=\pm \frac14+\frac{\tau_2}2$ and $z_3=\pm \frac14$, then the equation of $\hat X_\nu$ is satisfied for 
any  $z_1 \in E_1$, but  $2z_1=0$ hence $g_{11}$ fixes 32 smooth points on $\hat X_\nu$.

\noindent If $z_2=\pm \frac14$ and $z_3=\pm \frac14$ then the equation of $\hat X_\nu$ is
$\nu b_1\mathcal{L}_1(z_1)_1+\mathcal{L}_1(z_1)_0=0$, since $2z_1=0$ we get no solution if 
$\hat X_\nu $ is smooth $(\nu  \notin B)$ and one solution  (a node) if $\hat X_\nu$ is singular ($\nu  \in B$).

\noindent An analogous argument holds if $z_2=\pm \frac14+\frac{\tau_2}2$ and $z_3=\pm \frac14+\frac{\tau_3}2$.

\noindent 
Therefore,  if $\hat X_\nu$ is smooth $g_{11}$ fixes 32 isolated points, else $g_{11}$ fixes 32 smooth isolated points and  8 nodes.  
 
 \item[\underline{$g_{14}$}:]  
 since $Fix(g_{14})\cap Fix (g_{11})= \emptyset$,
the fixed locus of $g_{14}$ is independent from $\nu$.
It fixes 64 points on $E_1\times E_2\times E_3$, namely
$$z\in \left\{ \frac{1}{4} \left( \begin{array}{c}
\pm \tau_1\\
\pm\tau_2\\
\pm\tau_3
\end{array}
\right)
+ \frac{1}{2}(\mathbb Z/ 2\mathbb Z)^3
\right\}\,.$$
Observe that  $\mathcal L_k(\pm \frac{\tau_k}{4})=b_k$ and 
$\mathcal L_k(\pm\frac{\tau_k}{4}+\frac{1}{2})=-b_k$.
It is a straightforward computation to show that exactly 32 of them lie on $\hat X_\nu$.

\item[\underline{$g_{15}$}:]  
  since $Fix(g_{15})\cap Fix (g_{11})= \emptyset$,
the fixed locus of $g_{15} $ is independent from $\nu$.
It fixes 64 points on $E_1\times E_2\times E_3$, namely
$$z\in \left\{ \frac{1}{4} \left( \begin{array}{c}
\pm \tau_1\\
\pm(1+\tau_2)\\
\pm(1+\tau_3)
\end{array}
\right)
+ \frac{1}{2}(\mathbb Z/ 2\mathbb Z)^3
\right\}\,.$$
Observe now that 
$$\mathcal L_k\left(\frac{1}{4}+ \frac{\tau_k}{4}\right)^2=\mathcal L_k\left(\frac{1}{4}+ \frac{\tau_k}{4}+\frac{1}{2}\right)^2=-a_k\,,$$
whence $\{\mathcal L_k\left(\frac{1}{4}+ \frac{\tau_k}{4}\right),\mathcal L_ik\left(\frac{1}{4}+ \frac{\tau_k}{4}+\frac{1}{2}\right)\}=
\{\sqrt{-1}b_k, -\sqrt{-1}b_k\}$.  
It is a straightforward computation to show that exactly 32 of them lie on $\hat X_\nu$.
\end{itemize}
\end{proof}

\begin{lemma}\label{FixLoc2}
Let $S\rightarrow \hat S= \hat X /G$ be a regular generalized Burniat type surface with
$(\hat X,G)=(\hat X_\mu, G_2)$.
Let $g \in \mathcal G_1$ be an element fixing point on $X_\mu$, 
then its fixed locus $Fix(g)$ on $\hat X_\mu$  is as in Table \ref{EqFixLoc2}.

\begin{table}[!ht]
    \centering

\begin{tabular}{c|c}
&$Fix(g_i)$, $X_\mu$\\
\hline
$g_1:=(0,0,0,0,0,0,1,0,0)$&\multirow{3}{*}{4 genus 5 curves} \\
$g_2:=(0,0,0,1,0,0,0,0,0)$&\\
$g_3:=(1,0,0,0,0,0,0,0,0)$&\\ %8 ell. c. pw. int. in a nd
\hline
$g_4:=(0,0,0,1,0,0,1,0,0)$&\multirow{3}{*}{32 pt} \\
$g_5:=(1,0,0,0,0,0,1,0,0)$&\\
$g_6:=(1,0,0,1,0,0,0,0,0)$&\\
\hline
$g_7:=(0,0,0,0,0,1,0,0,1)$&\multirow{3}{*}{16 pt, 8 ell. curves}   \\
$g_8:=(0,0,1,0,0,0,0,0,1)$&\\
$g_9:=(0,0,1,0,0,1,0,0,0)$&\\
\hline
$g_{11}:=(1,0,0,0,0,1,0,0,1)$&\multirow{3}{*}{32  pt}\\
$g_{12}:=(0,0,1,1,0,0,0,0,1)$&\\
$g_{13}:=(0,0,1,0,0,1,1,0,0)$&\\
\end{tabular}\caption{}\label{EqFixLoc2}
\end{table}
\end{lemma}

\begin{proof} Noting that $g_0=(1,0,0,1,0,0,1,0,0)\in G_2$, the same arguments of the proof
of Lemma \ref{FixLoc1} hold, and the statement follows.

\end{proof}

Finally, we consider the case  $\hat X=\hat X_b$; we study the fixed locus only of the
 elements in $\mathcal G_0$ having fixed locus of dimension one on $T$, 
since it is enough for our purposes.

\begin{lemma}\label{FixLoc3}
Let $S\rightarrow \hat S= \hat X /G$ be a regular generalized Burniat type surface with
$(\hat X,G)=(\hat X_b, G_j)$, $j\in\{3,4\}$.
Let $g \in \mathcal G_0$ be an element having fixed locus of dimension one on $T$
then its fixed locus $Fix(g)$ on $\hat X_b$  is as in Table \ref{EqFixLoc3}.

\begin{table}[!ht]
    \centering

\begin{tabular}{c|c|c}
&$Fix(g_i)$, $X_b$ smooth &$Fix(g_i)$, $X_b$ singular\\
\hline
$g_4:=(0,0,0,1,0,0,1,0,0)$&\multirow{3}{*}{32 pt} &\multirow{3}{*}{16 pt, 8 nodes}  \\
$g_5:=(1,0,0,0,0,0,1,0,0)$&&\\
$g_6:=(1,0,0,1,0,0,0,0,0)$&&\\
\hline
$g_7:=(0,0,0,0,0,1,0,0,1)$&\multirow{3}{*}{16 pt, 8 ell. curves} &\multirow{3}{*}{8 nodes, 8 ell. curves}  \\
$g_8:=(0,0,1,0,0,0,0,0,1)$&\\
$g_9:=(0,0,1,0,0,1,0,0,0)$&\\
\end{tabular}\caption{}\label{EqFixLoc3}
\end{table}

\end{lemma}

\begin{proof} By Proposition \ref{NumberOfSing}, $\hat X_b$ is singular if and only if 
$b\in B'=\{\pm 1, \pm a_i, \pm a_i a_j, \pm a_1a_2a_3\} $,  with $ i\neq j\in\{1,2,3\}$;
in this case the eight nodes on $\hat X_b$ are fixed by $g_0=(1,0,0,1,0,0,1,0,0)$.

\begin{itemize}[leftmargin=.5cm]

\item[\underline{$g_4$}:]
it fixes the points $(z_1,\overline{z_2}, \overline{z_3})$  with 
$2 \overline{z_2}=2 \overline{z_3}= 0$.
If $b \notin B'$, for every pair $(\overline{z_2}, \overline{z_3})$,  $2z_1\neq 0$, whence $g_4$ fixes 32 
points on $\hat X_b$.

\noindent If $b\in B'$, for 8 choices of $(\overline{z_2}, \overline{z_3})$ there are two values
 of $z_1$ verifying the  equation of $\hat X_b$, while for the other 8 possibilities there is a unique 
 value of $z_1$ verifying the equation of $\hat X_b$,  whence $Fix(g_4)$
 is given by  16 smooth points and the 8 nodes of $\hat X_b$.

\item[\underline{$g_5$-$g_6$}:] these cases are analogous to $g_4$.
  
\item[\underline{$g_7$}:] 
 since $Fix(g_7)\cap Fix (g_0)= \emptyset$, the fixed locus of $g_7$ is independent from $b$.

The involution $g_7$  fixes the points $(z_1,\overline{z_2}, \overline{z_3})$  
with $2 \overline{z_2}=\frac12$, $2\overline{z_3}= \frac12$.
Let $(\mathcal{L}_i(z_i)_0:\mathcal{L}_i(z_i)_1) $ be the homogeneous coordinates of the point $\mathcal{L}_i(z_i)$. 
The equation of $\hat X_b$ is then 
$$\mathcal L_1(z_1)_0\mathcal L_2(z_2)_0\mathcal L_3(z_3)_0 = 
b_1b_2b_3 \mathcal L_1(z_1)_1\mathcal L_2(z_2)_1\mathcal L_3(z_3)_1\,.$$
It follows easily from the properties of the  Legendre $\mathcal L$-function that 
$ \left(\mathcal L_i\left(\frac{1}{4}\right)_0: \mathcal L_i\left(\frac{1}{4}\right)_1\right)=(0:1)\,, 
\left(\mathcal L_i\left(\frac{1}{4}+\frac{\tau_i}{2}\right)_0: \mathcal L_i\left(\frac{1}{4}+\frac{\tau_i}{2}\right)_1\right)=(1:0)\,.
$
If $z_2=\pm \frac14$ and $z_3=\pm \frac14+\frac{\tau_3}2$ or
 $z_2=\pm \frac14+\frac{\tau_2}2$ and $z_3=\pm \frac14$, then the equation is satisfied for 
any $z_1 \in E_1$, i.e $g_7$ fixes 8 disjoint elliptic curves.

\noindent If $z_2=\pm \frac14$ and $z_3=\pm \frac14$ then the equation becomes
$\mathcal{L}_1(z_1)_1=0$ that has two solutions.
If $z_2=\pm \frac14+\frac{\tau_2}2$ and $z_3=\pm \frac14+\frac{\tau_3}2$ 
 then the equation becomes $\mathcal{L}_1(z_1)_0=0$ that has two solutions, whence
 $g_7$ fixes 16 isolated points and 8 disjoint elliptic curves on $\hat X_b$.

\item[\underline{$g_8$-$g_9$}:] these cases are analogous to $g_7$.
\end{itemize}\end{proof}

We are now ready to prove our main result:
\begin{theorem}\label{mainthm}
Let $\epsilon \colon S\rightarrow \hat S= \hat X /G$ be a regular generalized Burniat type surface: 
 $(\hat{X},G) \in \{(\hat X_\nu,G_1),\, (\hat X_\mu,G_2),\,(\hat X_b,G_j), \,j=3,4\}$.
Then it verifies the Bloch conjecture.
\end{theorem}

We recall the following result, which allows to prove that each single surface in the moduli space corresponding to GBT surfaces of type $j$ ($1 \leq j \leq 4$) satisfies Boch's conjecture.
\begin{theorem}[{\cite{BCF14}}]\label{BCF}\quad
\begin{itemize}
\item[i)]Let $S$ be a smooth projective surface homotopically equivalent to a GBT surface  
$S_i$ of type $i$. Then $S$ is a GBT surface of type $i$, i.e. contained in the same irreducible family as $S_i$.
\item[ii)] The connected components $\mathfrak{N}_i$
of the Gieseker moduli space $\mathfrak{M}^{can}_{1,6}$ corresponding to GBT surfaces of type $i$
is irreducible, generically smooth, normal and unirational of dimension 4 ($i=1,2$) and
of dimension 3 else.
\end{itemize}
\end{theorem}
Together with Theorem \ref{mainthm} we thus obtain:

\begin{theorem}
Let $S$ be any surface such that its moduli point $[S]\in \mathfrak{N}_i$, $1 \leq i \leq 4$, then $S$ satisfies Bloch's conjecture.
\end{theorem}

\begin{proof}[Proof of Theorem \ref{mainthm}]
We prove the statement case by case. In each case we consider a group 
$H\cong (\mathbb Z/2\mathbb Z)^2 < \Aut(\hat X)$  which 
allow us define a group $H'\cong (\mathbb Z/2\mathbb Z)^2 < \Aut(S)$
satisfying the assumptions of Corollary \ref{corBl}.

$G_1$)   Let us consider the involutions
$\sigma_1:=(0,0,0,0,0,0,1,0,0)$ and $\sigma_2:=(0,0,0,1,0,0,0,0,0)$ in $\mathcal G'_1$.
%%%%%%%%%%%%%%%%%%%%%%%%%%%%%%%%%%%%%%%%%%%%%%%%%%%%%%%%

In the coset $\sigma_1G_1$ there are four elements fixing points on $\hat X_\nu$:
$$
\begin{array}{ccc}
g_1=(0,0,0,0,0,0,1,0,0)&\qquad & g_6=(1,0,0,1,0,0,0,0,0)\\
g_7=(0,0,0,0,0,1,0,0,1) &\qquad & g_{15}=(0,1,0,1,1,1,1,1,1)  
\end{array}
$$

\noindent Since $g_0:=(1,0,0,1,0,0,1,0,0)\in G_1$ has no fixed points on $\hat X_\nu$, 
the 4 sets $Fix_{\hat X_\nu}(g_k)$ ($k\in\{1,6,7,15\}$) are pairwise disjoint. 

\noindent $g_1$ fixes 4 curves of genus 5 on $\hat X_\nu$, $G_1$ acts transitively 
on this set of curves and 
 $g_0:=(1,0,0,1,0,0,1,0,0)\in G_1$ maps  each curve onto itself,
hence a genus 3 curve is fixed  by $\overline{\sigma_1}$  on $\hat S$.

\noindent $g_6$ and $g_{15}$ fix 32 points each on $\hat X_\nu$ and $G_1$ 
acts freely on these two sets:
we get  8 points fixed by $\overline{\sigma_1}$.

\noindent $g_7$ fixes 8 disjoint elliptic curves and 16 points if $\hat X_\nu$ is smooth, 
8 nodes otherwise.
Since $G_1$ acts freely on the set of points and transitively on the set of  curves, we get that
$\overline{\sigma_1}$ fixes one elliptic curve and either 2 points or 1 node.

\noindent It follows that the involution $\overline{\sigma_1} $ on $\hat S$ lifts
 to an involution $\sigma'_1$ on $S$ 
 whose fixed locus contains  a genus 3 curve, an elliptic curve and  8 isolated smooth points,
by Proposition \ref{NotGT}, $S/\sigma'_1$ is not of general type.

\

In the coset $\sigma_2G_1$ there are three elements fixing points on $\hat X_\nu$:
 $$
\begin{array}{ccc}
g_2=(0,0,0,1,0,0,0,0,0)&\qquad & g_5=(1,0,0,0,0,0,1,0,0)\\
 g_{11}=(1,0,0,0,0,1,0,0,1)  &&
\end{array}
$$
 \noindent The 3 sets $Fix_{\hat X_\nu}(g_k)$ ($k\in\{2,5,11\}$) are pairwise disjoint, 
 since $g_0$ has no fixed points on  $\hat X_\nu$.
 
\noindent $g_2$ fixes 4 curves of genus 5 on $\hat X_\nu$, $G_1$ acts transitively 
on this set of curves and 
 $g_0:=(1,0,0,1,0,0,1,0,0)\in G_1$ maps  each curve onto itself,
hence a genus 3 curve is fixed  by $\overline{\sigma_2}$  on $\hat S$. 
  
\noindent $g_5$  fixes 32 points on $\hat X_\nu$: we get  4 points fixed by $\overline{\sigma_2}$. 
 
\noindent If $\hat X_\nu$ is smooth $g_{11}$ fixes 32 isolated points, else 
 it fixes 32 smooth isolated points and  8 nodes: we get that 
$\overline{\sigma_2}$ fixes 4 smooth points and,  if $\hat S$ is singular, a node too.

\noindent It follows that the involution $\overline{\sigma_2} $ on $\hat S$ lifts
 to an involution $\sigma'_2$ on $S$ 
 whose fixed locus contains  a genus 3 curve, an elliptic curve and  8 isolated smooth points,
by Proposition \ref{NotGT}, $S/\sigma'_2$ is not of general type.

\

 Let $\sigma_3:= \sigma_1+\sigma_2$. In the coset $\sigma_3G_1$ there are
  three elements fixing points on $\hat X_\nu$:
  $$
\begin{array}{ccc}
g_3=(1,0,0,0,0,0,0,0,0)&\qquad & g_4=(0,0,0,1,0,0,1,0,0)\\
 g_{14}=(0,1,0,0,1,0,0,1,0)  &&
\end{array}
$$

 \noindent The 3 sets $Fix_{\hat X_\nu}(g_k)$ ($k\in\{3,4,14\}$) are pairwise disjoint, 
 since $g_0$ has no fixed points on  $\hat X_\nu$.
  
  \noindent $g_4$ and $g_{14}$ fix 32 points each on $\hat X_\nu$ and $G_1$ 
acts freely on these two sets: we get  8 points fixed by $\overline{\sigma_3}$.

 \noindent If $\hat X_\nu$ is smooth, $g_3$ fixes  four disjoint smooth curves of genus 5 
 on $\hat X_\nu$.
   Let $H:=\langle (1,0,0,1,0,0,1,0,0), (0,0,0,0,0,1,1,0,1) \rangle \triangleleft G_1$, 
   each curve is invariant under $H$, hence $\overline{\sigma_3}$ fixes two disjoint genus 2 curves.
 
\noindent If $\hat X_\nu$ is singular, $g_3$ fixes two disjoint smooth curves of genus 5 and
 8 elliptic curves such that a point belongs to two of them if and only if it is a node.
 Looking at the $G_1$ action on this configuration of curves, one can easily prove  that
 $\overline{\sigma_3}$ fixes on $\hat S$  two elliptic curves intersecting in a node and a genus 2 curve.

\noindent It follows that the involution $\overline{\sigma_3} $ on $\hat S$ lifts
 to an involution $\sigma'_3$ on $S$ 
 whose fixed locus contains  a genus 2 curve and  8 isolated smooth points,
by Proposition \ref{NotGT}, $S/\sigma'_3$ is not of general type.

 Applying Corollary \ref{corBl},  with $(\mathbb Z/2\mathbb Z)^2=
\langle \sigma_1, \sigma_2\rangle$,   we conclude that $S$ verifies Bloch's conjecture.

%%%%%%%%%%%%%%%%%%%%%%%%%%%%%%%%%%%%%%%%%%%%%%%%%%%%%%%%
%%%%%%%%%%%%%%%%%%%%%%%%%%%%%%%%%%%%%%%%%%%%%%%%%%%%%%%%
%%%%%%%%%%%%%%%%%%%%%%%%%%%%%%%%%%%%%%%%%%%%%%%%%%%%%%%%

\

$G_2$) Let us consider the involutions  $\sigma_4:=(1,0,0,0,0,0,0,0,0)$
and $\sigma_5:=(0,0,0,1,0,0,0,0,0)$ in $\mathcal{G}_1$.

In the coset $\sigma_4G_2$, there are four elements fixing points on $\hat X_\mu$:
  $$
\begin{array}{ccc}
g_3=(1,0,0,0,0,0,0,0,0)&\qquad & g_4=(0,0,0,1,0,0,1,0,0)\\
g_9=(0,0,1,0,0,1,0,0,0)&\qquad & g_{12}=(0,0,1,1,0,0,0,0,1)  
\end{array}
$$
Since $g_0:=(1,0,0,1,0,0,1,0,0)\in G_2$ has no fixed points,
 the 4 sets $Fix_{\hat X_\mu}(l_j)$ ($j\in\{3,4,9,12\}$) are pairwise disjoint.

\noindent $g_3$ fixes four genus 5 curves that are
invariant under the subgroup $\langle (0,0,0,1,0,1,0,0,1), (1,0,0,1,0,0,1,0,0) \rangle \triangleleft G_2$,
hence $\sigma'_4$ fixes two disjoint genus 2 curves.

\noindent $g_4$ and $g_{12}$ fix 32 points each on $\hat X_\nu$ and $G_2$ 
acts freely on these two sets: we get  8 points fixed by $\overline{\sigma_4}$. 

\noindent  $g_9$ fixes 16 isolated points  and 8 disjoint elliptic curves 
on $\hat X_\mu$, each curve is invariant under $(1,0,1,0,0,1,0,0,0)\in G_2$.
We get that $\sigma'_4$ fixes 2 isolated points and two disjoint elliptic curves.

It follows that $Fix({\sigma'_4})$ is given by 10 isolated fixed points, 
two genus 2 curves and two elliptic curves. 
By Proposition \ref{NotGT}, the quotient $S/{\sigma'_4}$ is not of general type.

The same argument shows that 
 $S/\sigma'_4$ and  $S/(\sigma_4+\sigma_5)'$ are not of general type, 
 whence $S$ verifies Bloch's conjecture, thanks to Corollary \ref{corBl}.

%%%%%%%%%%%%%%%%%%%%%%%%%%%%%%%%%%%%%%%%%%%%%%%%%%%%%%%%

\

$G_3$)  Let us consider the involutions
$\sigma_6:=(1,0,0,1,0,0,0,0,0)$ and $\sigma_7:=(1,0,0,0,0,0,1,0,0)$.
In the coset $\sigma_6G_3$, there are two elements fixing points on $\hat X$:
  $$
g_6=(1,0,0,1,0,0,0,0,0)\qquad  g_8=(0,0,1,0,0,0,0,0,1)
$$
The 2 sets $Fix_{\hat X_b}(h_j)$ ($j\in\{6,8\}$) are  disjoint.

\noindent $g_6$ fixes 32 points on $\hat X_b$ if it is smooth, 16 smooth points and 8 nodes
otherwise. Since $G_3$ acts freely on this set of points
we get that  $\overline{\sigma_6}$ fixes either 4  points ($\hat X_b$  smooth) 
or 2 smooth points and the node.

\noindent $g_8$ fixes 16 isolated smooth points and 8 elliptic curves.

\noindent 
In the smooth case $Fix({\sigma'_6})$ 
is given by  6 isolated fixed points and  an elliptic curve. By Proposition \ref{NotGT}, the quotient $S/\sigma'_6$ is not of general type.

 \noindent In the singular case,  $\overline{\sigma_6}$ fixes the node $p$, 
 4 smooth points and an elliptic curve.
  Let $\Gamma:=\epsilon^{-1}(p)$, the involution
$\overline{\sigma_6} $ lifts to an involution $\sigma'_6$ on $S$ such that
 $\sigma'_6(\Gamma)=\Gamma\cong \mathbb P^1$ and
it fixes an elliptic curve,  4 smooth points and either two isolated points on $\Gamma$ or  $\Gamma$: $(-2)$-curve.
In both cases, by Proposition \ref{NotGT}, $S/\sigma'_6$ is not of general type.

%%%%%%%%%%%%%%%%%%%%%%%%%%%%%%%%%%%%%%%%%%%%%%%%%%%%%%%%%%%%%%%%%%%%%%%%%%%%%%%%

\noindent The same argument shows that 
 $S/\sigma'_7$ and  $S/(\sigma_6+\sigma_7)'$ are not of general type,
  whence $S$ verifies Bloch's conjecture, thanks to Corollary \ref{corBl}.

\

$G_4$) Considering involutions $(1,0,0,1,0,0,0,0,0)$ and  $(1,0,0,0,0,0,1,0,0)$,
 this case is analogous to the $G_3$-case.

\end{proof}

%%%%%%%%%%%%%%%%%%%%%%%%%BIBLIOGRAPHY%%%%%%%%%%%%%%%%%%%%%%%%%%%%
\bibliographystyle{alpha}

\end{document}